\documentclass[12pt]{amsart}
\usepackage{amscd,amsmath,amsthm,amssymb}
\usepackage{tikz}
\tikzstyle{punkt}=[circle, fill=black, minimum size=1mm,inner sep=0pt, draw]

\usepackage{amsfonts,amssymb,amscd,amsmath,enumerate,verbatim}

\usetikzlibrary{arrows}
\input xy
\xyoption{all}
\def\frk{\mathfrak}               

\def\Phi{{\frk N}}
\def\opn#1#2{\def#1{\operatorname{#2}}} 
\opn\chara{char} \opn\length{\ell} \opn\pd{pd} \opn\rk{rk}
\opn\projdim{proj\,dim} \opn\injdim{inj\,dim} \opn\rank{rank}
\opn\depth{depth} \opn\grade{grade} \opn\height{height}
\opn\embdim{emb\,dim} \opn\codim{codim}

\opn\Tr{Tr} \opn\bigrank{big\,rank}
\opn\superheight{superheight}\opn\lcm{lcm}
\opn\trdeg{tr\,deg}
\opn\reg{reg} \opn\lreg{lreg} \opn\ini{in} \opn\lpd{lpd}
\opn\size{size}\opn{\mult}{mult}
\opn\div{div} \opn\Div{Div} \opn\cl{cl} \opn\Cl{Cl}
\opn\Spec{Spec} \opn\Supp{Supp} \opn\supp{supp} \opn\Sing{Sing}
\opn\Ass{Ass} \opn\Min{Min}
\opn\Ann{Ann} \opn\Rad{Rad} \opn\Soc{Soc}
\opn\Syz{Syz} \opn\Im{Im} \opn\Ker{Ker} \opn\Coker{Coker}
\opn\Am{Am} \opn\Hom{Hom} \opn\Tor{Tor} \opn\Ext{Ext}
\opn\End{End} \opn\Aut{Aut} \opn\id{id} \opn\ini{in}

\opn\nat{nat}
\opn\pff{pf}
\opn\Pf{Pf} \opn\GL{GL} \opn\SL{SL} \opn\mod{mod} \opn\ord{ord}
\opn\Gin{Gin}
\opn\Hilb{Hilb}\opn\adeg{adeg}\opn\std{std}\opn\ip{infpt}
\opn\Pol{Pol}
\opn\sat{sat}
\opn\Var{Var}
\opn\Gen{Gen}

\opn\aff{aff} \opn\con{conv} \opn\relint{relint} \opn\st{st}
\opn\lk{lk} \opn\cn{cn} \opn\core{core} \opn\vol{vol}
\opn\link{link} \opn\star{star}
\opn\gr{gr}

\def\pot#1#2{#1[\kern-0.28ex[#2]\kern-0.28ex]}

\opn\dirlim{\underrightarrow{\lim}}
\opn\inivlim{\underleftarrow{\lim}}
\def\Implies{\ifmmode\Longrightarrow \else
        \unskip${}\Longrightarrow{}$\ignorespaces\fi}
\def\implies{\ifmmode\Rightarrow \else
        \unskip${}\Rightarrow{}$\ignorespaces\fi}
\def\iff{\ifmmode\Longleftrightarrow \else
        \unskip${}\Longleftrightarrow{}$\ignorespaces\fi}

\let\:=\colon
\newtheorem{Theorem}{Theorem}[section]

\newtheorem{Corollary}[Theorem]{Corollary}

\newtheorem{Remark}[Theorem]{Remark}

\newtheorem{Example}[Theorem]{Example}

\let\epsilon\varepsilon
\let\phi=\varphi
\let\kappa=\varkappa
\def\qed{\ifhmode\textqed\fi
      \ifmmode\ifinner\quad\qedsymbol\else\dispqed\fi\fi}
\def\textqed{\unskip\nobreak\penalty50
       \hskip2em\hbox{}\nobreak\hfil\qedsymbol
       \parfillskip=0pt \finalhyphendemerits=0}
\def\dispqed{\rlap{\qquad\qedsymbol}}

\opn\dist{dist}
\def\pnt{{\raise0.5mm\hbox{\large\bf.}}}

\opn\Lex{Lex}
\opn\diam{diam}

\begin{document}
\title{Diameter and connectivity of finite simple graphs~II}
\author[T.~Hibi]{Takayuki Hibi}
\address[Takayuki Hibi]
{Department of Pure and Applied Mathematics, 
Graduate School of Information Science and Technology, 
Osaka University, 
Suita, Osaka 565-0871, Japan}
\email{hibi@math.sci.osaka-u.ac.jp}
\author[S.~Saeedi~Madani]{Sara Saeedi Madani}
\address[Sara Saeedi Madani]
{Department of Mathematics and Computer Science, Amirkabir University of Technology, Tehran, Iran, and School of Mathematics, Institute for Research in Fundamental Sciences, Tehran, Iran} 
\email{sarasaeedi@aut.ac.ir}

\begin{abstract}
	Let $G$ be a finite simple non-complete connected graph on $[n] = \{1, \ldots, n\}$ and $\kappa(G) \geq 1$ its vertex connectivity.  Let $f(G)$ denote the number of free vertices of $G$ and $\diam(G)$ the diameter of $G$. The final goal of this paper is to determine all sequences of integers $(n,f,d,k)$ with $n\geq 8$, $f\geq 0$, $d\geq 2$ and $k\geq 1$ for which there exists a finite simple non-complete connected graph on $[n]$ with $f=f(G)$, $d=\diam (G)$ and $k=\kappa(G)$.
\end{abstract}
\maketitle
\section*{Introduction}

In this paper, we are interested in the diameter, connectivity and the number of free vertices of a finite simple non-complete connected graph and how they are related to each other. Even though the approach of this paper is purely combinatorial, our initial motivation has arisen from the combinatorial study of binomial edge ideals \cite{HHHKR, Ohtani}, where these graphical invariants are closely related to a classical invariant in commutative algebra.  

Let $G$ be a finite simple connected graph on the vertex set $[n] = \{1, \ldots, n\}$ with $n\geq 3$ and the edge set $E(G)$.  The {\em distance} $\dist_G(i,j)$ of $i$ and $j$ in $[n]$ with $i \neq j$ is the smallest length, (i.e. the number of edges) of paths connecting $i$ and $j$ in $G$.  Especially, if $\{i, j\} \in E(G)$, then $\dist_G(i,j) = 1$.  The {\em diameter} of $G$ is 
\[
\diam(G) = \max \{\dist_G(i,j) : i, j \in [n] \}.
\]

The {\em induced subgraph} of $G$ on $T \subset [n]$ is the subgraph $G_T$ of $G$ on the vertex set $T$ whose edges are those $\{i, j\} \in E(G)$ with $ i \in T$ and $j \in T$.  Let $\kappa(G) \geq 1$ denote the {\em vertex connectivity} of $G$.  In other words, $\kappa(G)$ is the smallest cardinality of $T \subset [n]$ for which 
$G_{[n]\setminus T}$ is disconnected.  

A vertex $i \in [n]$ is {\em free} if $\{i, j\} \in E(G)$ and $\{i, j'\} \in E(G)$ with $j \neq j'$, yield $\{j, j'\} \in E(G)$. In particular, a free vertex belongs to exactly one maximal clique (i.e., complete subgraph) of $G$. If $\Delta(G)$ is the clique complex of $G$, i.e. the simplicial complex whose faces are the vertex sets of the cliques of $G$, then a vertex is a free vertex of $G$ if and only if it belongs to exactly one facet of $\Delta(G)$. Such a vertex is sometimes called a \emph{simplicial vertex}. If a free vertex has degree equal to~$1$, then it is just the same as a \emph{leaf} in classical graph theory. Free vertices of graphs play an important role in the literature of monomial and binomial edge ideals, especially in the case of chordal graphs, see for example \cite{HH}. Let $f(G)$ denote the number of free vertices of $G$.

In the study of binomial edge ideals \cite{HHHKR, Ohtani}, it was shown in \cite{BN, RSK} that, when $G$ is non-complete, one has
\begin{eqnarray}
\label{inequality}
f(G) + \diam(G) \leq n + 2 - \kappa(G).
\end{eqnarray}

Our final goal is to determine the possible sequences $(n, f, d, k)$ of integers for which there is a finite simple non-complete connected graph $G$ on $[n]$ with $f = f(G)$, $d = \diam(G)$ and $k = \kappa(G)$.  Toward our goal, in the previous paper \cite{partI}, the possible sequences $(n, f, d, k)$ of integers for which there is a finite simple non-complete connected graph $G$ on $[n]$ with $f = f(G)$, $d = \diam(G)$ and $k = \kappa(G)$ satisfying $f + d = n + 2 - k$ was determined.  Furthermore, finite simple non-complete connected graphs $G$ on $[n]$ satisfying $f(G) + \diam(G) = n + 2 - \kappa(G)$ were classified. To simplify the notation, we let 
\[
\phi(G)=f(G)+\diam (G)+ \kappa(G)
\]
and 
\[
\delta(G)=(f(G),\diam (G),\kappa(G)).
\]
The present paper is a continuation of \cite{partI}.  First, in Section~\ref{bound and minimum}, a purely combinatorial proof of (\ref{inequality}) is given and also a classification of graphs $G$ with $\phi(G)=3$, i.e. the minimum possible value for $\phi(G)$, is provided.  Second, in Section~\ref{possible valus}, for each integer $4 \leq i \leq n + 1$, a finite simple non-complete connected graph $G$ on $[n]$ with $\phi(G) = i$ is constructed. Moreover, given an integer $n\geq 8$, all possible sequences $\delta=(f, d, k)$ of integers for which there is a finite simple non-complete connected graph $G$ on $[n]$ with $\delta=\delta(G)$ and $\phi(G)=4, 5, 6$ are determined. In Section~\ref{complete list}, based on the computations done in the previous sections, we provide the complete list of all possible tripes $\delta(G)$, which is our main goal in this paper. Finally, we reprove the classifications of $\delta(G)$ with $\phi(G)=n+1, n+2$, given in \cite{Indian} and \cite{partI}, respectively. 

Throughout the present paper, all graphs are finite and simple, and we denote the vertex set and the edge set of a finite simple graph $G$ by $V(G)$ and $E(G)$, respectively, unless we mention something else.

\section{Bounds on $\phi(G)$ and its minimum possible value}\label{bound and minimum}

In this section, first we give an independent and totally combinatorial proof for the inequality~(\ref{inequality}). We need to recall some notation and terminologies from graph theory. 

Recall that a connected graph $G$ with $n$ vertices is called $k$-{\em connected} for an integer $k$ with $1\leq k<n$, if for each $T\subseteq V(G)$ with $|T|<k$, the induced subgraph $G_{[n]\setminus T}$ of $G$ is connected. In particular, any connected graph $G$ is $\kappa(G)$-connected, but not $(\kappa(G)+1)$-connected.   

Let $P:u_1,u_2,\ldots, u_{\ell}$ be a path of length $\ell-1$ in a graph. Then we say that $P$ is a $(u_1,u_{\ell})$-path and $u_2,\ldots,u_{\ell-1}$ are {\em internal} vertices of $P$. Then, two paths $P:u_1,u_2,\ldots, u_{\ell}$ and $Q:v_1,v_2,\ldots, v_s$ are called {\em internally disjoint} if they do not have any common internal vertex. It is a well known fact in graph theory that a connected graph $G$ is $k$-connected if and only if for any two distinct vertices $u$ and $v$ there exist at least $k$ internally disjoint $(u,v)$-paths in $G$.    
 
\begin{Theorem}\label{Combinatorial proof}
\label{Theorem}
Let $G$ be a finite simple non-complete connected graph on $[n]$. Then 
\[
f(G) + \diam(G) \leq n + 2 - \kappa(G).
\] 
In particular, 
\[
3\leq \phi(G)\leq n+2.
\]
\end{Theorem}

\begin{proof}
For simplicity, let $f=f(G)$, $d=\diam(G)$ and $k=\kappa(G)$. Since $G$ is connected, we have $k\geq 1$, and hence $d-2\leq k(d-2)$. Therefore, we have 
\begin{equation}\label{d-2+k}
d-2+k\leq k(d-2)+k=k(d-1).
\end{equation}	
Now, let $u$ and $v$ be two vertices of $G$ such that $\dist_G(u,v)=d$. Since $G$ is $k$-connected, there exist at least $k$ induced internally disjoint paths between $u$ and $v$ of length at least~$d$. In particular, all internal vertices of these paths are not free. Thus, the number of non-free vertices of $G$ is at least~$k(d-1)$. Therefore, we have 
\[
n\geq f+k(d-1) \geq f+d-2+k
\]
where the last inequality follows from~(\ref{d-2+k}). Hence we get the desired inequality. The lower bound~$3$ is also achieved by $f(G)=0$, $\diam(G)=2$ and $\kappa(G)=1$, where the last two follow since $G$ is non-complete and connected, respectively.  
\end{proof}

In \cite{partI} all graphs $G$ on $n$ vertices for which $\phi(G)=n+2$ were characterized. Namely, the case in which $\phi(G)$ is maximized has been characterized. It is now interesting to characterize the case in which $\phi(G)$ is minimized, i.e. $\phi(G)=3$.

First we recall the {\em join product} of two vertex disjoint graphs $G_1$ and $G_2$ which is denoted by $G_1*G_2$. The vertex set of $G_1*G_2$ is just $V(G_1)\cup V(G_2)$ and its edge set is 
\[
E(G_1*G_2)=E(G_1)\cup E(G_2)\cup \{\{i,j\}: i\in V(G_1), j\in V(G_2)\}.
\]

Also recall that a vertex $v$ of a connected graph $G$ is called a {\em cut vertex} if the graph $G-v$ is disconnected. In this paper, we denote the complete graph with $n$ vertices by $K_n$.

\begin{Theorem}\label{phi=3}
 Let $G$ be a finite simple non-complete connected graph on $[n]$. Then the following hold:
 \begin{enumerate}
 	\item Let $n<9$. Then there is no graph $G$ with $\phi(G)=3$.
 	\item Let $n\geq 9$. Then $\phi(G)=3$ if and only if $G=K_1*H$ where $H$ is a disconnected graph with no free vertex. 
 \end{enumerate} 
\end{Theorem} 

\begin{proof}
Let $G$ be a finite simple non-complete connected graph with $\phi(G)=3$, which implies that $f(G)=0$, $\kappa(G)=1$ and $\diam(G)=2$. Since $\kappa(G)=1$, $G$ has a cut vertex $v$. Then $G-v$ is a disconnected graph with the connected components $H_1,\ldots,H_s$ with $s\geq 2$. Since $\diam(G)=2$, it follows that any vertex in $H_i$ for $i=1,\ldots,s$ is adjacent to the vertex $v$ in $G$. This implies that $G=v*(G-v)$. If $H_i$ for some $i=1,\ldots,s$ has a free vertex, then according to the construction of the join product, this vertex is a free vertex in $G$ as well, which is a contradiction to the assumption $f(G)=0$. Thus, none of $H_i$'s has a free vertex, and hence $|V(H_i)|\geq 4$ for each $i=1,\ldots,s$. This argument in particular implies that for $n<9$, there is no desired graph and for $n\geq 9$ the statement~(2) of the theorem holds.   
\end{proof}

\section{Possible values of $\phi(G)$}\label{possible valus}

After considering the minimum and maximum possible values for $\phi(G)$, it natural to ask if any possible number between $3$ and $n+2$ can be achieved by the invariant $\phi(G)$ for some graph $G$. Motivated by a conjecture on Cohen-Macaulay property of binomial edge ideals of graphs posed by Bolognini et al. in \cite{BMS}, and the fact that any connected graph with Cohen-Macaulay binomial edge ideal has connectivity equal to~1, we provide the desired graphs $G$ to answer the aforementioned question with $\kappa(G)=1$.    

Let $G$ and $H$ be two graphs on disjoint sets of vertices. Then we denote the {\em disjoint union} of $G$ and $H$ by $G\cup H$ which is the graph on the vertex set $V(G)\cup V(H)$ and the edge set $E(G)\cup E(H)$.  

Let $G$ be a graph on $[n]$. Then we say that a graph $G'$ is obtained by {\em duplicating} a vertex, say~$n$, if $V(G')=V(G)\cup \{n'\}$ where $n'\notin V(G)$ and 
\[
E(G')=E(G)\cup \{\{j,n'\}:j\in N_G(n)\}
\] where $N_G(n)$ is the set of all neighbors, i.e.  the adjacent vertices, of $n$ in $G$.  

\begin{Theorem}\label{any number}
	Given any integer $n\geq 8$, for any integer $i$ with $4\leq i\leq n+1$, there exists a finite simple non-complete connected graph $G$ on $[n]$ such that $\kappa(G)=1$ and $\phi(G)=i$.  
\end{Theorem}

\begin{proof}
	First assume that $i=4$. Let $C_{n-2}$ be the $(n-2)$-cycle. Then, let $G=K_1*(C_{n-2}\cup K_1)$. Then $f(G)=1$, $\diam(G)=2$, $\kappa(G)=1$ and $\phi(G)=4$. 
	
	Next, suppose that $i=5$. Let $C_{n-3}$ be the $(n-3)$-cycle. Then, let $G=K_1*(C_{n-3}\cup K_2)$. Then $f(G)=2$, $\diam(G)=2$, $\kappa(G)=1$ and $\phi(G)=5$.   
	 
	Now, assume that $6\leq i\leq n-2$. Let $P_i:1\ldots,i$ be the path graph. Let $G$ be the graph on $[n]$ obtained from $P_i$ by duplicating the vertex~$2$ and $n-i-1$ times duplicating the vertex $i-1$ in $P_i$.
	Then $f(G)=0$, $\diam(G)=i-1$, $\kappa(G)=1$ and $\phi(G)=i$. 
	
	Next assume that $i=n-1$. Then, let $G$ be the graph depicted in Figure~\ref{diam=n-3}. Then $f(G)=1$, $\diam(G)=n-3$, $\kappa(G)=1$ and $\phi(G)=n-1$.
	
	Now suppose that $i=n$. Then, let $G$ be the graph depicted in Figure~\ref{diam=n-2}. Then $f(G)=1$, $\diam(G)=n-2$, $\kappa(G)=1$ and $\phi(G)=n$.
	
	Finally, assume that $i=n+1$. Then, let $G$ be the graph depicted in Figure~\ref{diam=n-1}. Then $f(G)=2$, $\diam(G)=n-2$, $\kappa(G)=1$ and $\phi(G)=n+1$.      
\end{proof}

\begin{figure}
	\centering
	\begin{tikzpicture}[scale=0.8]
\draw [line width=1.5pt] (-4,2.63)-- (-4,1);
\draw [line width=1.5pt] (-4,1)-- (-6,1);
\draw [line width=1.5pt] (-6,1)-- (-6.02,2.63);
\draw [line width=1.5pt] (-4,2.63)-- (-4,1);
\draw [line width=1.5pt] (-4,1)-- (-6,1);
\draw [line width=1.5pt] (-4,1)-- (-2,1);
\draw [line width=1.5pt] (-2,1)-- (0,1);
\draw [line width=1.5pt] (2,1)-- (4,1);
\draw [line width=1.5pt] (-5,4)-- (-6.02,2.63);
\draw [line width=1.5pt] (-5,4)-- (-4,2.63);
\begin{scriptsize}
	\draw [fill=black] (-6.02,2.63) circle (2.5pt);
	\draw [fill=black] (-4,2.63) circle (2.5pt);
	\draw [fill=black] (-4,1) circle (2.5pt);
	\draw [fill=black] (-6,1) circle (2.5pt);
	\draw [fill=black] (-2,1) circle (2.5pt);
	\draw [fill=black] (0,1) circle (2.5pt);
	\draw [fill=black] (2,1) circle (2.5pt);
	\draw [fill=black] (4,1) circle (2.5pt);
	\draw [fill=black] (0.64,0.99) circle (1pt);
	\draw [fill=black] (1,0.99) circle (1pt);
	\draw [fill=black] (1.32,0.99) circle (1pt);
	\draw [fill=black] (-5,4) circle (2.5pt);
\end{scriptsize}
\end{tikzpicture}
\caption{The graph $G$ on $[n]$ with $\delta(G)=(1,n-3,1)$}
\label{diam=n-3}
\end{figure}
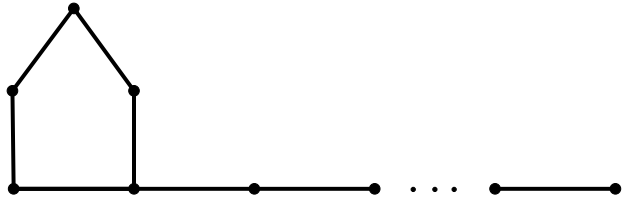

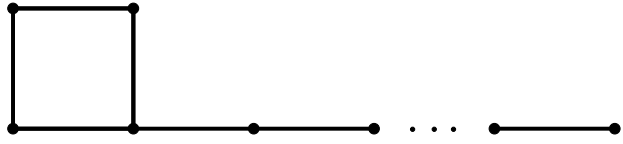
\begin{figure}
	\centering
	\begin{tikzpicture}[scale=0.8]
	\draw [line width=1.5pt] (-6,3)-- (-4,3);
	\draw [line width=1.5pt] (-4,3)-- (-4,1);
	\draw [line width=1.5pt] (-4,1)-- (-6,1);
	\draw [line width=1.5pt] (-6,1)-- (-6,3);
	\draw [line width=1.5pt] (-6,3)-- (-4,3);
	\draw [line width=1.5pt] (-4,3)-- (-4,1);
	\draw [line width=1.5pt] (-4,1)-- (-6,1);
	\draw [line width=1.5pt] (-4,1)-- (-2,1);
	\draw [line width=1.5pt] (-2,1)-- (0,1);
	\draw [line width=1.5pt] (2,1)-- (4,1);
	\begin{scriptsize}
	\draw [fill=black] (-6,3) circle (2.5pt);
	\draw [fill=black] (-4,3) circle (2.5pt);
	\draw [fill=black] (-4,1) circle (2.5pt);
	\draw [fill=black] (-6,1) circle (2.5pt);
	\draw [fill=black] (-2,1) circle (2.5pt);
	\draw [fill=black] (0,1) circle (2.5pt);
	\draw [fill=black] (2,1) circle (2.5pt);
	\draw [fill=black] (4,1) circle (2.5pt);
	\draw [fill=black] (0.64,0.99) circle (1pt);
	\draw [fill=black] (1,0.99) circle (1pt);
	\draw [fill=black] (1.32,0.99) circle (1pt);
	\end{scriptsize}
	\end{tikzpicture}
\caption{The graph $G$ on $[n]$ with $\delta(G)=(1,n-2,1)$}
\label{diam=n-2}
\end{figure}

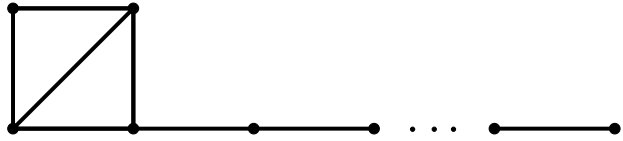
\begin{figure}
	\centering
	\begin{tikzpicture}[scale=0.8]
	\draw [line width=1.5pt] (-6,3)-- (-4,3);
	\draw [line width=1.5pt] (-4,3)-- (-4,1);
	\draw [line width=1.5pt] (-4,1)-- (-6,1);
	\draw [line width=1.5pt] (-6,1)-- (-6,3);
	\draw [line width=1.5pt] (-6,3)-- (-4,3);
	\draw [line width=1.5pt] (-4,3)-- (-4,1);
	\draw [line width=1.5pt] (-4,1)-- (-6,1);
	\draw [line width=1.5pt] (-4,1)-- (-2,1);
	\draw [line width=1.5pt] (-2,1)-- (0,1);
	\draw [line width=1.5pt] (2,1)-- (4,1);
	\draw [line width=1.5pt] (-4,3)-- (-6,1);
	\begin{scriptsize}
	\draw [fill=black] (-6,3) circle (2.5pt);
	\draw [fill=black] (-4,3) circle (2.5pt);
	\draw [fill=black] (-4,1) circle (2.5pt);
	\draw [fill=black] (-6,1) circle (2.5pt);
	\draw [fill=black] (-2,1) circle (2.5pt);
	\draw [fill=black] (0,1) circle (2.5pt);
	\draw [fill=black] (2,1) circle (2.5pt);
	\draw [fill=black] (4,1) circle (2.5pt);
	\draw [fill=black] (0.64,0.99) circle (1pt);
	\draw [fill=black] (1,0.99) circle (1pt);
	\draw [fill=black] (1.32,0.99) circle (1pt);
	\end{scriptsize}
		\end{tikzpicture}
	\caption{The graph $G$ on $[n]$ with $\delta(G)=(2,n-1,1)$}
	\label{diam=n-1}
	\end{figure}

In the proof of Theorem~\ref{any number}, in the case of $i=4$, we provided a graph $G$ with $n\geq 8$ vertices with $\delta(G)=(1,2,1)$.
In the next example, we also provide graphs $G$ with all other possible  triples $\delta(G)=(f(G),\diam(G),\kappa(G))$ which yield $\phi(G)=4$.   
	
\begin{Example}\label{i=4}
\begin{enumerate}
	\item Let $n\geq 8$ and let $C_{n-4}$ be the $(n-4)$-cycle on $[n-4]$ and $P_3:n-3,n-2,n-1$ be the path graph. Then, let $H=n*(C_{n-4}\cup P_3)$ and let $G$ be the graph on $[n]$ obtained from $H$ by removing the edge $\{n-2,n\}$. 
	Then $\delta(G)=(0,3,1)$ and $\phi(G)=4$. 
	\item Let $n\geq 4$ and let $G=K_{2,n-2}$ be the complete bipartite graph. Then $\delta(G)=(0,2,2)$ and $\phi(G)=4$. 
\end{enumerate}
\end{Example}

In the proof of Theorem~\ref{any number}, in the case of $i=5$, we provided a graph $G$ on $n\geq 8$ vertices with $\delta(G)=(2,2,1)$. In the next example, we also provide graphs $G$ with all other possible triples $\delta(G)$ which yield $\phi(G)=5$.   

\begin{Example}\label{i=5}
	\begin{enumerate}
		\item Let $n\geq 6$ and let $K_{2,n-4}$ be the complete bipartite graph on the bipartition of its vertex set $\{1,2\}\cup \{3,4,\ldots,n-2\}$. Then we set $G$ to be the graph on $[n]$ obtained by adding the edges $\{1,n-1\}$, $\{1,n\}$, $\{2,n\}$ and $\{n-1,n\}$ to $K_{2,n-4}$. Then $\delta(G)=(1,2,2)$ and $\phi(G)=5$. 
		
		\item Let $n\geq 5$ and let $K_{2,n-4}$ be the complete bipartite graph on the bipartition of its vertex set $\{1,2\}\cup \{3,4,\ldots,n-2\}$. Then we set $G$ to be the graph on $[n]$ obtained by adding the edges $\{1,n\}$, $\{2,n\}$ and $\{n-1,n\}$ to $K_{2,n-4}$. Then $\delta(G)=(1,3,1)$ and $\phi(G)=5$. 
		
		\item Let $n\geq 6$ and let $G=K_{3,n-3}$ be the complete bipartite graph. Then $\delta(G)=(0,2,3)$ and $\phi(G)=5$.
		
		\item Let $n\geq 7$ and let $K_{2,n-6}$ be the complete bipartite graph on the bipartition of its vertex set $\{1,2\}\cup \{3,4,\ldots,n-4\}$. Then we set $G$ to be the graph on $[n]$ obtained by adding the edges $\{1,n-2\}$, $\{2,n-2\}$, $\{n-3,n-2\}$, $\{n-2,n-1\}$, $\{n-1,n\}$ and $\{n-3,n\}$ to $K_{2,n-6}$. Then $\delta(G)=(0,4,1)$ and $\phi(G)=5$.
		
		\item Let $n\geq 6$ and let $K_{2,n-5}$ be the complete bipartite graph on the bipartition of its vertex set $\{1,2\}\cup \{3,4,\ldots,n-3\}$. Then we set $G$ to be the graph on $[n]$ obtained by adding the edges $\{1,n-1\}$, $\{2,n-2\}$, $\{n-1,n\}$ and $\{n-2,n\}$ to $K_{2,n-3}$. Then $\delta(G)=(0,3,2)$ and $\phi(G)=5$.  
	\end{enumerate}
\end{Example}

In the proof of Theorem~\ref{any number}, in the case of $i=6$, we provided a graph $G$ with $n\geq 8$ vertices with $\delta(G)=(0,5,1)$. In the next example, we also provide graphs $G$ with all other possible triples $\delta(G)$ which yield $\phi(G)=6$.

\begin{Example}\label{i=6}
	\begin{enumerate}
		\item Let $n\geq 8$ and let $G=K_1*(K_3\cup C_{n-4})$. Then $\delta(G)=(3,2,1)$ and $\phi(G)=6$. 
		
		\item Let $n\geq 8$ and let $H=1*C_{n-4}$ on the vertex set $\{1,2,\ldots,n-3\}$. Then we set $G$ to be the graph on $[n]$ obtained by adding the edges $\{1,n-2\}$, $\{n-2,n\}$, $\{n-2,n-1\}$ and $\{n-1,n\}$ to $H$. Then $\delta(G)=(2,3,1)$ and $\phi(G)=6$.  
		
		\item Let $n\geq 8$ and let $G=K_2*(C_{n-4}\cup K_1\cup K_1)$.
		Then $\delta(G)=(2,2,2)$ and $\phi(G)=6$.  
		
        \item Let $n\geq 8$ and let $H=1*C_{n-4}$ with the vertex set $\{1,2,\ldots,n-3\}$. Then we set $G$ to be the graph on $[n]$ obtained by adding the edges $\{1,n-2\}$, $\{n-2,n-1\}$ and $\{n-1,n\}$ to $H$. Then $\delta(G)=(1,4,1)$ and $\phi(G)=6$. 
        
        \item Let $n\geq 8$ and let $G=K_3*(C_{n-4}\cup K_1)$.
        Then $\delta(G)=(1,2,3)$ and $\phi(G)=6$.
        
        \item Let $n\geq 6$ and let $K_{2,n-5}$ be the complete bipartite graph with the bipartition of its vertex set $\{1,2\}\cup \{3,4,\ldots,n-3\}$. Then we set $G$ to be the graph on $[n]$ obtained by adding the edges $\{1,n-1\}$, $\{2,n-2\}$, $\{n-2,n-1\}$, $\{n-2,n\}$ and $\{n-1,n\}$ to $H$. Then $\delta(G)=(1,3,2)$ and $\phi(G)=6$.
        
        \item Let $n\geq 8$ and let $G=K_{4,n-4}$. Then $\delta(G)=(0,2,4)$ and $\phi(G)=6$.
        
        \item Let $n\geq 8$ and let $C_8$ be the $8$-cycle on $[8]$. Then we set $G$ to be the graph on $[n]$ obtained by $n-8$ times duplicating the vertex~$1$. 
        Then $\delta(G)=(0,4,2)$ and $\phi(G)=6$.
        
        \item Let $n\geq 8$ and let $K_{3,n-7}$ be the complete bipartite graph with the bipartition of its vertex set $\{1,2,3\}\cup \{4,\ldots,n-4\}$. Then we set $G$ to be the graph on $[n]$ obtained by adding the edges $\{n-3,n-2\}$, $\{n-2,n-1\}$, $\{n-j,n\}$ and $\{j,n-j\}$ for each $j=1,2,3$ to $H$. Then $\delta(G)=(0,3,3)$ and $\phi(G)=6$.
	\end{enumerate}
\end{Example}

\section{The complete list of possible $\delta(G)$}\label{complete list}

In this section we make the complete list of all possible $\delta(G)$. First we start with the case $\kappa(G)=1$.  

\begin{Theorem}\label{kappa=1}
 Fix integers $n\geq 8$, $f\geq 0$ and $d\geq 2$. Then there exists a finite simple non-complete connected graph $G$ on $[n]$ with $\delta(G)=(f,d,1)$ if and only if one of the following conditions is satisfied:   
 \begin{enumerate}
 	\item $f=0$ and $2\leq d\leq n-3$;
 	\item $f=1$  and $2\leq d\leq n-2$;
 	\item $f\geq 2$, $d\geq2$ and  $f+d\leq n+1$.
 \end{enumerate}
\end{Theorem}

\begin{proof}
 Note that if $f(G)=0$, then $\diam(G)$ is never equal to~$n-1$ or $n-2$. Also, if $f(G)=1$, then $\diam(G)\neq n-1$. Thus, if $\delta(G)=(f,d,1)$ for some non-complete connected graph $G$, then one of the conditions~(1), (2), (3) holds. For the converse, in the following, in each case~(1), (2), (3) we construct a desired graph. 
 
 (1) Assume that $f=0$ and $2\leq d\leq n-3$. If $d=2,3$, then $\phi(G)=3,4$ which was discussed in Theorem~\ref{phi=3} and Example~\ref{i=4}. Now, let $4\leq d\leq n-3$. Let $P_{d+1}: 1,2, \ldots, d+1$ be the path graph. Then, let $G$ be the graph obtained by duplicating the vertex~$2$ and $n-d-2$ times duplicating the vertex~$d$. Then, $\delta(G)=(0,d,1)$.
 
 (2) Assume that $f=1$  and $2\leq d\leq n-2$. If $d=2$, then $\phi(G)=4$ which was discussed in Theorem~\ref{any number}. So, let $3\leq d\leq n-2$. Let $P_{d+1}: 1,2, \ldots, d+1$ be the path graph. Then, let $G$ be the graph obtained by $n-d-1$ times duplicating the vertex~$2$. Then, $\delta(G)=(1,d,1)$. 
 
 (3) Assume that $f\geq 2$, $d\geq2$ and  $f+d\leq n+1$. First assume that $f+d=n+1$. Then, let $P_{d+1}:1,2,\ldots, d+1$ be the path graph and let $G$ be the graph obtained from $P_{d+1}$ by $f-2$ times duplicating the vertex $d+1$. Then $\delta(G)=(f,d,1)$. This case was also considered in \cite[Theorem~1.4]{partI}. So, we assume that $f+d\leq n$. 
 
 If $f=2$ and $d=2$, then $\phi(G)=5$ which was discussed in Theorem~\ref{any number}. Now, let $f\geq 3$ and $d=2$. Let $K_{1,n-1}$ be the complete bipartite graph with the bipartition of its vertex set $\{1\}\cup \{2, 3, \ldots,n\}$. Let $G$ be the graph on $[n]$ which is obtained from $K_{1,n-1}$ by adding the edges $\{j,j+1\}$ for each $j=2,\ldots, n-f+1$. Then, $\delta(G)=(f,2,1)$. 
 
 Now, let $f\geq 2$, $d\geq3$ and  $f+d\leq n-2$. Let $P_{d+1}:1,2,\ldots,d+1$ be the path graph. Then, let $H$ be the graph on $[d+2]$ which is obtained by adding the edges $\{d,d+2\}$ and $\{d+1,d+2\}$ to $P_{d+1}$. Now, let $G$ be the graph on $[n]$ obtained from $H$ by $f-1$ times duplicating the vertex $d+2$ and by $n-d-f-1$ times duplicating the vertex~$2$ in $H$. Then, $\delta(G)=(f,d,1)$.  
    
 Next, let $f\geq 2$, $d\geq3$ and  $f+d=n-1$. Let $P_{d+1}:1,2,\ldots,d+1$ be the path graph. Then, let $G$ be the graph on $[n]$ obtained from $P_{d+1}$ by $f-1$ times duplicating the vertex $d+1$ and by duplicating the vertex~$2$. Then, $\delta(G)=(f,d,1)$.       
 
 Finally, let $f\geq 2$, $d\geq3$ and  $f+d=n$. Let $P_{d+1}:1,2,\ldots,d+1$ be the path graph. Then, let $H$ be the graph on $[d+2]$ obtained from $P_{d+1}$ by duplicating~$2$ and adding the edge $\{2,d+2\}$. Then, let $G$ be the graph on $[n]$ obtained from $H$ by $f-2$ times duplicating the vertex $d+1$. Then, $\delta(G)=(f,d,1)$.        
\end{proof}

Next, we consider the case $\diam(G)=2$.  

\begin{Theorem}\label{diam=2}
	Fix integers $n\geq 8$, $f\geq 0$ and $k\geq 2$. Then there exists a finite simple non-complete connected graph $G$ on $[n]$ with $\delta(G)=(f,2,k)$ if and only if one of the following conditions is satisfied:   
	\begin{enumerate}
		\item $f=0$ and $2\leq k\leq n-2$;
		\item $f=1$  and $2\leq k\leq n-3$;
		\item $f\geq 2$, $k\geq2$ and  $f+k\leq n$.
	\end{enumerate}
\end{Theorem}

\begin{proof}
Note that for any non-complete connected graph $G$, we have $\kappa(G)\leq n-2$. If $f(G)=1$, then $\kappa(G)\leq n-3$. In fact, if $\kappa(G)=n-2$, then each vertex of $G$ is adjacent to at least $n-2$ other vertices in $G$. Suppose that the vertex~$1$ is free. If $1$ is adjacent to all other vertices of $G$, then $G$ must be complete graph, since $1$ is a free vertex. Thus, we may assume that there is no edge connecting the vertices~$1$ and~$2$ in $G$. This implies that $2$ is adjecent to all vertices other than~$1$ which means that $2$ is also a free vertex of $G$. This contradicts the assumption that $f(G)=1$. Thus, if $\delta(G)=(f,2,k)$ for some non-complete connected graph $G$, then one of the conditions~(1), (2), (3) holds. For the converse, in the following, in each case~(1), (2), (3) we construct a desired graph. 

(1) Assume that $f=0$ and $2\leq k\leq n-5$. If $k=2$, then $\phi(G)=4$ and the result follows from Example~\ref{i=4}. So, let $3\leq k\leq n-5$. Let $K_k$ be the complete graph on $[k]$, and let $P_{n-k-2}: k+1,k+2, \ldots, n-2$ and $P_2:n-1,n$ be path graphs. Then, let $H=K_k*(P_{n-k-2}\cup P_2)$. Then, let $G$ be the graph on $[n]$ obtained by removing the edges $\{1,k+2\}$, $\{2,n-3\}$, $\{1,n-1\}$ and $\{2,n\}$ from $H$. Then, $\delta(G)=(0,2,k)$. Note that if $n-k=5$, then the vertices $k+2$ and $n-3$ coincide. 

Now, assume that $f=0$ and $k= n-4$. Let $K_{n-4}$ be the complete graph on $[n-4]$, let $P_2: n-3,n-2$ and $P'_2:n-1,n$ be path graphs. Then, let $H=K_k*(P_2\cup P'_2)$. Then, let $G$ be the graph on $[n]$ obtained by removing the edges $\{1,n-3\}$, $\{2,n-2\}$, $\{1,n-1\}$ and $\{2,n\}$ from $H$. Then, $\delta(G)=(0,2,n-4)$. 

Next, assume that $f=0$ and $k= n-3$. Let $K_{n-3}$ be the complete graph on $[n-3]$ and let $P_3:n-2, n-1, n$ be the path graph. Let $H=K_{n-3}*P_3$. Then let $G$ be the graph on $[n]$ obtained by removing the edges $\{1,n-2\}$, $\{2,n-1\}$, and $\{3,n\}$ from $H$. Then, $\delta(G)=(0,2,n-3)$.    

Next, assume that $f=0$ and $k= n-2$. Let $K_{n-2}$ be the complete graph on $[n-2]$ and let $P_2:n-1, n$ be the path graph. Let $H=K_{n-2}*P_2$. Then let $G$ be the graph on $[n]$ obtained by removing the edges $\{1,n-1\}$ and $\{2,n\}$ from $H$. Then, $\delta(G)=(0,2,n-2)$. 

(2) Assume that $f=1$  and $2\leq k\leq n-4$. If $k=2$, then $\phi(G)=5$, and hence the result follows from Example~\ref{i=5}. Now suppose that $3\leq k\leq n-4$. Let $K_k$ be the complete graph on $[k]$ and let $P_{n-k-1}: k+1,k+2, \ldots, n-1$ be the path graph. Then, let $H=K_k*(P_{n-k-1}\cup n)$. Then, let $G$ be the graph on $[n]$ obtained by removing the edges $\{1,k+2\}$ and $\{2,n-2\}$ from $H$. Then, $\delta(G)=(1,2,k)$. Note that if $n-k=4$, then the vertices $k+2$ and $n-2$ coincide.

Now, assume that $f=1$ and $k= n-3$. Let $K_{n-3}$ be the complete graph on $[n-3]$, let $P_2: n-2,n-1$ be the path graph. Then, let $H=K_{n-3}*(P_2\cup n)$. Then, let $G$ be the graph on $[n]$ obtained by removing the edges $\{1,n-2\}$ and $\{2,n-1\}$ from $H$. Then, $\delta(G)=(1,2,n-3)$.  

(3) First assume that $f=2$ and $2\leq k\leq n-4$. Then let $K_k$ be the complete graph on $[k]$, let $P_{n-k-2}: k+1,k+2,\ldots,n-2$ and $P_2: n-1,n$ be the path graphs. Then, let $H=K_k*(P_{n-k-2}\cup P_2)$. Then, let $G$ be the graph on $[n]$ obtained by removing the edges $\{1,n-1\}$ and $\{2,n\}$ from $H$. Then, $\delta(G)=(2,2,k)$. 

Now, assume that $f=2$ and $k=n-3$. Then let $K_{n-3}$ be the complete graph on $[n-3]$ and let $P_{3}: n-2, n-1, n$ be the path graph. Then, let $H=K_{n-3}*P_3$. Then, let $G$ be the graph on $[n]$ obtained by removing the edges $\{n-1,n-2\}$ and $\{1,n\}$ from $H$. Then, $\delta(G)=(2,2,n-3)$.

Next, assume that $f=2$ and $k=n-2$. Then, let $G=K_{n-2}*(K_1\cup K_1)$. Then, $\delta(G)=(2,2,n-2)$.    

Finally, assume that $f\geq 3$ and $2\leq k\leq n-f$. Then let $K_k$ be the complete graph on $[n]\setminus [n-k]$ and let $C_{n-k}$ be the $(n-k)$-cycle on the vertex set $[n-k]$. Let $H=K_k*C_{n-k}$. Then let $G$ be the graph on $[n]$ obtained by removing the edges $\{j,j+1\}$ for each $j=1,\ldots, f-1$ from $H$. Then, $\delta(G)=(f,2,k)$. 
\end{proof}

Finally, we consider $\kappa(G)\geq 2$ and $\diam (G)\geq 3$. 

\begin{Theorem}\label{main}
 Fix integers $f\geq 0$, $d\geq 3$ and $k\geq 2$. Then there exists a finite simple non-complete connected graph $G$ on $[n]$ with $\delta(G)=(f,d,k)$ if and only if 
 \[
 n\geq k(d-1)+\max \{2,f\}.
 \]
\end{Theorem}

\begin{proof}
First let $G$ be a non-complete connected graph on $[n]$ with $\delta(G)=(f,d,k)$ where $f\geq 0$, $d\geq 3$ and $k\geq 2$. Let $Q_1:1,2,\ldots,d+1$ be a longest path in $G$. Since $G$ is $k$-connected, there exist $k$ internally disjoint induced paths between the vertices $1$ and $d+1$, say $Q_1,\ldots, Q_k$. Since $d=\diam (G)$, each $Q_i$ for $i=2,...,k$ has length at least~$d$. Since none of the internal vertices of each $Q_i$ for $i=1,\ldots,k$ is a free vertex, it follows that the number of non-free vertices of $G$ is at least $k(d-1)$. Note that the vertices $1$ and $d+1$ can be free or not. Then, we deduce that
\[
n\geq k(d-1)+\max \{2,f\}. 
\]     

Conversely, suppose that $n\geq k(d-1)+\max \{2,f\}$. Let $Q_1, Q_2,\ldots, Q_k$ be path graphs of length~$d-2$ on disjoint sets of vertices. Let $i$ and $i+k$ be the leaves of $Q_i$ for each $i=1,\ldots,k$. Let $H=Q_1\cup Q_2\cup \ldots \cup Q_k$ with $V(H)=[k(d-1)]$. Then, let $H'$ be the graph on $[k(d-1)+2]$ obtained by adding the edges $\{j,k(d-1)+1\}$ and $\{k+j,k(d-1)+2\}$ for each $j=1,\ldots, k$. 

First assume that $f=0$. Then, let $G$ be the graph on $[n]$ obtained by $n-(k(d-1)+2)$ times duplicating the vertex~$1$ in $H'$. Then, $\delta(G)=(0,d,k)$.

Next assume that $f=1$. Then, let $G$ be the graph on $[n]$ obtained by $n-(k(d-1)+2)$ times duplicating the vertex~$1$ in $H'$ and by adding the edges $\{k+i,k+j\}$ for all $1\leq i<j \leq k$. Then, the vertex $k(d-1)+2$ is the only free vertex of $G$, and hence $\delta(G)=(1,d,k)$. 

Finally, assume that $f\geq 2$. Let $e=\{s,t\}$ be an edge of $Q_k$. Then, let $H''$ be the graph on $[k(d-1)+3]$ by adding the edge $\{s,k(d-1)+3\}$ and $\{t,k(d-1)+3\}$ to $H'$. Then let $G$ be the graph on $[n]$ obtained from $H''$ by $f-3$ times duplicating the vertex $k(d-1)+3$ and $n-(k(d-1)+f)$ times duplicating the vertex~$1$ and then by connecting all the neighbors of the vertex $k(d-1)+1$ in the new graph to each other by an edge as well as connecting all the neighbors of the vertex $k(d-1)+2$ to each other. Then, $\delta(G)=(f,d,k)$.      
\end{proof}

In the next two corollaries we consider, in particular, the characterizations due to $\phi(G)=n+2$ and $\phi(G)=n+1$ to recover \cite[Theorem~1.4]{partI} and \cite[Theorem~3.3]{Indian}, respectively.
  
\begin{Corollary}\label{Ours}
Fix integers $n\geq 8$, $f\geq 0$, $d\geq 2$ and $k\geq 1$. Then there exists a finite simple non-complete connected graph $G$ on $[n]$ with $\phi(G)=n+2$ and $\delta(G)=(f,d,k)$ if and only if $f+d+k=n+2$ and one of the following conditions is satisfied:   
\begin{enumerate}
	\item $f\geq 2$, $k\geq 1$ and $d=2$;
	\item $f\geq 2$, $k=1$ and $d\geq 2$.
\end{enumerate}
\end{Corollary}

\begin{proof}
 First suppose that $G$ is a graph on $[n]$ with $\phi(G)=n+2$, $f=f(G)\geq 0$, $k=\kappa(G)\geq 2$ and $d=\diam (G)\geq 3$. Then by Theorem~\ref{main}, we have 
 \[
 f+d+k-2\geq k(d-1)+\max \{2,f\}.
 \] 
 Thus, we have
 \[
 (k-1)(d-2)\leq f-\max \{2,f\},
 \]  
 which is impossible. Therefore, if $\phi(G)=n+2$, then $k=1$ or $d=2$, and hence the result follows from Theorem~\ref{kappa=1} and Theorem~\ref{diam=2}.  
\end{proof}

\begin{Corollary}\label{Indian people}
	Fix integers $n\geq 8$, $f\geq 0$, $d\geq 2$ and $k\geq 1$. Then there exists a finite simple non-complete connected graph $G$ on $[n]$ with $\phi(G)=n+1$ and $\delta(G)=(f,d,k)$ if and only if $f+d+k=n+1$ and one of the following conditions is satisfied:   
	\begin{enumerate}
		\item $f\geq 2$, $k\geq 1$ and $d=2$;
		\item $f\geq 2$, $k=1$ and $d\geq 2$;
		\item $f\geq 2$, $k=2$ and $d=3$.
	\end{enumerate}
\end{Corollary}

\begin{proof}
	First suppose that $G$ is a graph on $[n]$ with $\phi(G)=n+1$, $f=f(G)\geq 0$, $k=\kappa(G)\geq 2$ and $d=\diam (G)\geq 3$. Then by Theorem~\ref{main}, we have 
	\[
	f+d+k-1\geq k(d-1)+\max \{2,f\}.
	\] 
	Thus, we have
	\[
	(k-1)(d-2)\leq f-\max \{2,f\}+1,
	\]  
	which is possible if and only if $f\geq 2$, $d=3$ and $k=2$. Therefore, if $\phi(G)=n+1$, then $k=1$ or $d=2$ or $f\geq 2$, $d=3$ and $k=2$. Then the result follows from Theorem~\ref{kappa=1}, Theorem~\ref{diam=2} and Theorem~\ref{main}.  
\end{proof}

\section*{Acknowledgment}

The present paper was completed while the authors stayed at Mathematisches Forschungsinstitut in Oberwolfach, August~18 to~30, 2024, in the frame of the Oberwolfach Research Fellows. Sara Saeedi Madani was in part supported by a grant from IPM (No. 1403130020).


\begin{thebibliography}{10}

\bibitem{BN} 
A. Banerjee and L. N\'u\~nez-Betancourt, ``Graph connectivity and binomial edge ideals'', Proc. Amer. Math. Soc. 145 (2017), 487-499.



\bibitem{BMS} D. Bolognini, A. Macchia, F. Strazzanti, {\em Cohen--Macaulay binomial edge ideals and accessible graphs}, J. Algebraic Combin. 55, (2022), 1139--1170.




\bibitem{HH}
J.~Herzog and T.~Hibi, ``Monomial Ideals'', GTM 260, Springer, 2011.

\bibitem{HHHKR} 
J. Herzog, T. Hibi, F. Hreinsd{\'o}ttir, T. Kahle and J. Rauh, ``Binomial edge ideals and conditional independence statements'', Adv. Appl. Math. 45 (2010), 317-333.

\bibitem{partI}
T. Hibi and S. Saeedi Madani, ``Diameter and connectivity of finite
simple graphs'', Mediterr. J. Math. 20 (2023), \#310, 11 pp.

\bibitem{Indian}  
A. V. Jayanthan and R. Sarkar, ``Depth of binomial edge ideals in terms of diameter and vertex connectivity'',  arXiv:2112.04835v3.



\bibitem{Ohtani} M. Ohtani, ``Graphs and ideals generated by some 2-minors'', Comm. Algebra. 39 (2011), 905-917.

\bibitem{RSK} 
M. Rouzbahani Malayeri, S. Saeedi Madani and D. Kiani, ``On the depth of binomial edge ideals of graphs'', J. Algebraic Combin., 55 (2022), 827-846.  

\end{thebibliography}
\end{document}